\documentclass[a4paper,11pt]{amsart}

\usepackage{amsmath,amsthm,amsfonts,amssymb}
\usepackage[bookmarks=false,colorlinks=true,
  linkcolor=blue,citecolor=blue,urlcolor=blue]{hyperref}
\usepackage{jesse-layout}

\newcommand{\SL}{\operatorname{SL}}
\newcommand{\bB}{\mathbb B}
\newcommand{\bK}{\mathbb K}
\newcommand{\bS}{\mathbb S}
\newcommand{\cU}{\mathcal U}

\theoremstyle{plain}
\newtheorem{theorem}{Theorem}[section]
\newtheorem{proposition}[theorem]{Proposition}
\newtheorem{lemma}[theorem]{Lemma}

\theoremstyle{definition}

\title[An ICC property-\texorpdfstring{$(T)$}{(T)} group]
{An ICC Group with Property \texorpdfstring{$(T)$}{(T)}\\
That Is Not Properly Proximal}
\author{Yanyu Wang}
\address{Department of Mathematics, Vanderbilt University, 1326 Stevenson
Center, Nashville, TN 37240, USA}
\email{yanyu.wang@vanderbilt.edu}
\date{}
\hypersetup{pdftitle={An ICC Group with Property (T) That Is Not Properly Proximal}}

\begin{document}
\begin{abstract}
We exhibit explicit ICC subgroups of
$\SL_7(\mathbb Q)$ with Kazhdan's property $(T)$ that is not properly
proximal.
\end{abstract}

\maketitle

\section{Introduction}\label{sec:introduction}

Proper proximality was introduced by Boutonnet, Ioana, and Peterson
\cite{BIP} to extend boundary methods in the study of group von Neumann
algebras beyond the setting of biexact groups. A principal motivation
was to study rigidity properties of the group von Neumann algebras
associated with higher-rank lattices, such as $\SL_d(\mathbb Z)$ for
$d\geq3$. These lattices are properly proximal
\cite[Proposition~4.14]{BIP} and have Kazhdan's property $(T)$
\cite{BHV}. This suggests asking whether a group with property $(T)$
can fail to be properly proximal.

Groups with property $(T)$ and infinite center provide easy examples
\cite{Cornulier2007}: they are inner amenable and hence are not properly
proximal \cite[Proposition~4.11]{BIP}. More generally, one can ask whether
inner amenability is the only obstruction to proper proximality.
Ishan, Peterson, and Ruth \cite{IPR} answered this negatively: they
proved that proper proximality is invariant under measure equivalence
and used this to construct non-inner amenable groups that are not
properly proximal.

ICC groups with property $(T)$ are non-inner amenable
\cite[Theorem~11]{AW}. It is therefore natural to ask whether there
exists an ICC group with property $(T)$ that is not properly proximal.
Peterson posed this question in his 2020 COSy talk notes
\cite[p.~6]{Peterson2020} and in his 2026 ICM proceedings paper
\cite[p.~756]{Peterson2026}. We give an affirmative answer with the
following explicit matrix groups.

\begin{theorem}\label{thm:main}
Fix a prime $p$, put $R=\mathbb Z[1/p]$, and let
\[
\Gamma=
\left\{
\begin{pmatrix}
h&a&z\\
0&k&y\\
0&0&1
\end{pmatrix}
:
\begin{array}{l}
h\in\SL_3(\mathbb Z),\quad k\in\SL_3(R),\\
a\in M_3(R),\quad y,z\in R^3
\end{array}
\right\}.
\]
Then $\Gamma$ is countable, has Kazhdan's property $(T)$, is ICC, and is not properly proximal.
\end{theorem}

Since discrete groups with property $(T)$ are finitely generated, our
examples are non-inner amenable finitely generated subgroups of
$\SL_7(\mathbb Q)$ that are not properly proximal. Thus
Theorem~\ref{thm:main} also answers negatively the question of
Boutonnet, Ioana, and Peterson asking whether every non-inner amenable
finitely generated linear group is properly proximal
\cite[Question~2(3)]{BIP}.

The failure of proper proximality will follow from the next observation,
which combines the state-compression arguments of
\cite[Section~8]{DKEP}.

\begin{proposition}\label{prop:central-semidir-obstruction}
Let $H\curvearrowright G$ be an action by automorphisms of countable
discrete groups, and let $W\leq G$ be an infinite $H$-invariant subgroup.
Assume that
\[
L(G)=W^*\!\left(L(W),L(W)'\cap L(G)\right).
\]
If the induced action $H\curvearrowright L(W)$ is not properly
proximal, then $G\rtimes H$ is not properly proximal.
\end{proposition}

We use proper proximality of actions in the sense of \cite{DKEP}.
The generation hypothesis holds whenever $G=W C_G(W)$, where $C_G(W)$
is the centralizer of $W$ in $G$. In particular, it is automatic if
$W\leq Z(G)$. In our example, $W$ is central in the property $(T)$ group
$G$, while the action of $H$ makes $G\rtimes H$ ICC and retains the
obstruction to proper proximality.

\medskip
\noindent\textbf{AI use statement.} The construction underlying the main result of this paper was found mainly
by GPT-5.6 Sol, under the author's mathematical guidance. The resulting
material was subsequently reorganized, revised, and checked by the author.
All mathematical statements, proofs, references, and final editorial
decisions remain the sole responsibility of the author.

\medskip
\noindent\textbf{Acknowledgements.}
The author would like to thank Changying Ding for suggesting this problem and Jesse Peterson for helpful discussions and valuable comments on an earlier draft of this paper.

\section{Proof of the main theorem}\label{sec:proof}

We first establish the state criterion needed for
Proposition~\ref{prop:central-semidir-obstruction}. We use the
terminology of \cite{BIP,DKEP}. Let $(P,\tau)$ be a separable finite
von Neumann algebra with a normal faithful tracial state, and let
$J_P$ denote the canonical conjugation on $L^2(P)$. For
$T\in\bB(L^2(P))$, put
\[
\|T\|_{\infty,1;P}
=\sup_{a,b\in(P)_1}
\bigl|\langle T\widehat a,\widehat b\rangle\bigr|,
\qquad
\bK^{\infty,1}(P)
=\overline{\bK(L^2(P))}^{\|\cdot\|_{\infty,1;P}},
\]
where the closure is taken inside $\bB(L^2(P))$, and define
\[
\bS(P)=\{T\in\bB(L^2(P)):
[T,J_PxJ_P]\in\bK^{\infty,1}(P)\text{ for every }x\in P\}.
\]
The algebra $P$ is properly proximal if $\bS(P)$ has no $P$-central
state whose restriction to $P$ is normal
\cite[Theorem~6.2 and the subsequent definition]{DKEP}.
Centrality means $\varphi(xT)=\varphi(Tx)$ for $x\in P$ and
$T\in\bS(P)$.

Let a countable group $\Lambda$ act on $(P,\tau)$ by trace-preserving
automorphisms $\alpha_s$, implemented by the Koopman unitaries
$U_s\widehat x=\widehat{\alpha_s(x)}$. The action is weakly compact
if there is a $\Lambda$-invariant, $P$-central state $\Omega$ on
$\bB(L^2(P))$ with $\Omega|_P=\tau$, where invariance means
$\Omega(U_sTU_s^*)=\Omega(T)$. Following \cite[Section~8]{DKEP},
the action is properly proximal if neither of the following holds:
\begin{enumerate}
\item[(A)] there is a nonzero $\Lambda$-invariant projection
      $p\in Z(P)$ such that the action on
      $(Pp,\tau(p)^{-1}\tau|_{Pp})$ is weakly compact;
\item[(B)] there is a $\Lambda$-invariant, $P$-central state
      $\varphi$ on $\bS(P)$ such that $\varphi|_P$ is normal and
      $\varphi|_{\bK^{\infty,1}(P)}=0$.
\end{enumerate}

\begin{lemma}\label{lem:pp-diffuse-state}
Let $\Lambda\curvearrowright(B,\tau)$ be a trace-preserving action
on a separable diffuse finite von Neumann algebra. The action is
not properly proximal if and only if there is a
$\Lambda$-invariant, $B$-central state $\varphi$ on $\bS(B)$ such
that $\varphi|_B$ is normal and
$\varphi|_{\bK^{\infty,1}(B)}=0$.
\end{lemma}

\begin{proof}
Only the implication from condition (A) to condition (B) requires
proof. Choose a nonzero invariant projection $p\in Z(B)$ for which
the action on $Q=Bp$, equipped with
$\tau_Q=\tau(p)^{-1}\tau|_Q$, is weakly compact.
By \cite[Definition~3.1 and Proposition~3.2]{OP}, there is a net of unit
vectors
$\eta_i\in L^2(Q\,\overline\otimes\,\overline Q)_+$ such that
\[
\|(v\otimes\overline v)\eta_i-\eta_i\|_2\longrightarrow0,
\qquad
\|(U_s\otimes\overline{U_s})\eta_i-\eta_i\|_2\longrightarrow0,
\]
for $v\in\cU(Q)$ and $s\in\Lambda$, and
$\langle(x\otimes1)\eta_i,\eta_i\rangle=\tau_Q(x)$ for $x\in Q$.
Here $U_s$ denotes the Koopman unitary on $L^2(Q)$.
Let $\Omega$ be a weak$^*$ cluster point of the states
\[
\Omega_i(T)=\langle(T\otimes1)\eta_i,\eta_i\rangle,
\qquad T\in\bB(L^2(Q)).
\]
The two invariance conditions make $\Omega$ $Q$-central and
$\Lambda$-invariant. Since
$J_Q\otimes J_{\overline Q}$ fixes each $\eta_i$, we also have
\[
\Omega_i(J_Qx^*J_Q)
=\overline{\Omega_i(x^*)}
=\tau_Q(x)
\qquad(x\in Q).
\]
Thus $\Omega$ restricts to the canonical normal traces on both
$Q$ and $J_QQJ_Q$.

The restriction of $\Omega$ to $\bK(L^2(Q))$ is represented by a
positive trace-class operator $D$. Centrality gives $D\in Q'$.
Diffuseness provides a sequence of unitaries $v_n\in\cU(Q)$ converging to $0$ weakly.
Since $D$ is compact and lies in $Q'$, we obtain
\[
\|D\xi\|_2=\|Dv_n\xi\|_2\longrightarrow0
\qquad(\xi\in L^2(Q)).
\]
Hence $D=0$, so $\Omega$ vanishes on $\bK(L^2(Q))$. Its normality
on $Q$ and $J_QQJ_Q$ now allows us to apply
\cite[Theorem~5.6]{DKEP}, giving
$\Omega|_{\bK^{\infty,1}(Q)}=0$.

Identifying \(L^2(Q,\tau_Q)\) with \(pL^2(B,\tau)\), the compression argument in [DKEP23, Lemma 8.7] (see also Remark 6.3 there) shows that \(T\mapsto pTp\) maps \(\mathbb K^{\infty,1}(B)\) into \(\mathbb K^{\infty,1}(Q)\) and \(\mathbb S(B)\) into \(\mathbb S(Q)\).

The state
\[
\varphi(T)=\Omega(pTp),
\qquad T\in\bS(B),
\]
is $B$-central and $\Lambda$-invariant, vanishes on
$\bK^{\infty,1}(B)$, and satisfies
$\varphi(b)=\tau(pb)/\tau(p)$ for $b\in B$. This proves (B).
\end{proof}

\begin{proof}[Proof of Proposition~\ref{prop:central-semidir-obstruction}]
Put $P=L(G)$ and $B=L(W)$. Since $W$ is infinite, $B$ is diffuse.
Lemma~\ref{lem:pp-diffuse-state} gives an $H$-invariant,
$B$-central state $\varphi$ on $\bS(B)$ that is normal on $B$ and
vanishes on $\bK^{\infty,1}(B)$.

Let $e_B:L^2(P)\to L^2(B)$ be the orthogonal projection, and let
$E_B:P\to B$ be the trace-preserving conditional expectation. By
\cite[Lemma~8.7]{DKEP}, the formula
\[
\psi(T)=\varphi(e_BTe_B),
\qquad T\in\bS(P),
\]
defines an $H$-invariant state vanishing on
$\bK^{\infty,1}(P)$, and
$\psi|_P=(\varphi|_B)\circ E_B$ is normal. The compression argument
in the second half of the proof of \cite[Theorem~8.8]{DKEP}
shows that $\psi$ is central over $B$ and $B'\cap P$. As
explicitly observed in the proof of \cite[Theorem~8.10]{DKEP},
the hypothesis $P=W^*(B,B'\cap P)$ therefore makes $\psi$
$P$-central.

Applying the proof of \cite[Proposition~8.11]{DKEP} to $\psi$
gives a $(P\rtimes H)$-central state on $\bS(P\rtimes H)$ whose
restriction to $P\rtimes H$ is normal. Thus
$L(G\rtimes H)\cong P\rtimes H$ is not properly proximal.
By \cite[Theorem~6.4]{DKEP}, neither is $G\rtimes H$.
\end{proof}

\begin{proof}[Proof of Theorem~\ref{thm:main}]
The matrix description makes $\Gamma$ a countable subgroup of
$\SL_7(\mathbb Q)$. Let $G$ be its subgroup defined by $h=I_3$, and
identify $H=\SL_3(\mathbb Z)$ with the subgroup
$\{\operatorname{diag}(h,I_3,1):h\in H\}$. Then $\Gamma=G\rtimes H$.
The group $G$ is the three-block example in
\cite[Definition~2.4 and Remark~2.6]{Cornulier2007}, with
$(n_1,n_2,n_3,n_4)=(3,3,0,1)$ and the empty block omitted.
The proof of \cite[Proposition~2.7]{Cornulier2007}, applied as in
Remark~2.6 there, gives property $(T)$ for $G$.
Since $H$ has property $(T)$ \cite[Example~1.7.4(i)]{BHV}, so does
$\Gamma$ by stability under extensions
\cite[Proposition~1.7.6]{BHV}.

To see that $\Gamma$ is ICC, suppose that $g\in\Gamma$ has a finite
conjugacy class. Its centralizer has finite index, so it contains
$I_7+mE_{ij}$ for some nonzero integer $m$, whenever $i\ne j$ lie
in the same one of the two $3\times3$ diagonal blocks; here $E_{ij}$
are the $7\times7$ matrix units. Thus $gE_{ij}=E_{ij}g$ for all such
$i,j$. Commuting with these matrix units forces
$g=\operatorname{diag}(\lambda I_3,\mu I_3,1)$.
Since $\lambda,\mu\in\mathbb Q$ and
$\lambda^3=\mu^3=1$, we obtain $g=1$.

It remains to prove failure of proper proximality. Put
\[
W_0=\left\{w_z=
\begin{pmatrix}I_3&0&z\\0&I_3&0\\0&0&1\end{pmatrix}
:z\in R^3\right\}.
\]
Then $W_0\cong(R^3,+)$ is infinite and central in $G$.
Conjugation by $h\in H$ sends $w_z$ to $w_{hz}$, so $W_0$ is
$H$-invariant. Under Fourier transform, the action on $L(W_0)$
is the dual action
$\SL_3(\mathbb Z)\curvearrowright\widehat{(\mathbb Z[1/p])^3}$. This action is not properly proximal by an observation in the paragraph following \cite[Prop~8.11]{DKEP}, using Theorem 8.12 there.
Indeed, set
$\Lambda=(\mathbb Z[1/p])^3\rtimes \SL_3(\mathbb Z)$. The diagonal embedding realizes $\Lambda$
as a lattice in
\[
(\mathbb R^3\rtimes\SL_3(\mathbb R))
\times(\mathbb Q_p^3\rtimes\SL_3(\mathbb Z_p)).
\]
The second factor is noncompact, amenable and unimodular, hence is
measure equivalent to $\mathbb Z$ by
\cite[Theorem~4.1]{KKR}. Since $\mathbb Z^3\rtimes \SL_3(\mathbb Z)$ is a lattice
in the first factor, it follows that
\[
\Lambda\sim_{\mathrm{ME}}(\mathbb Z^3\rtimes \SL_3(\mathbb Z))\times\mathbb Z.
\]
The group on the right has infinite center, so it is inner amenable
and not properly proximal \cite[Proposition~4.11]{BIP}.
Measure-equivalence invariance of proper proximality
\cite[Theorem~1.3]{IPR} therefore shows that $\Lambda$ is not
properly proximal. As $H=\SL_3(\mathbb Z)$ is properly proximal
\cite[Proposition~4.14]{BIP},
\cite[Theorems~6.4 and~8.12]{DKEP} imply that
$H\curvearrowright L(W_0)$ is not properly proximal.
The centrality of $W_0$ gives
$L(G)=W^*(L(W_0),L(W_0)'\cap L(G))$, so
Proposition~\ref{prop:central-semidir-obstruction} shows that
$\Gamma$ is not properly proximal.
\end{proof}

\end{document}